\documentclass[12pt,a4paper,reqno]{amsart}
\allowdisplaybreaks
\usepackage{amsmath}
\usepackage{amsfonts}
\usepackage{amssymb,amsthm,amsfonts,amsthm,latexsym,enumerate,url,cases}
\numberwithin{equation}{section}
\usepackage{mathrsfs}
\usepackage{hyperref}
\hypersetup{colorlinks=true,citecolor=blue,linkcolor=blue,urlcolor=blue}
\def\pmod #1{\ ({\rm{mod}}\ #1)}

\theoremstyle{plain}
\newtheorem{theorem}{Theorem}
\newtheorem*{theorem0}{Nathanson's theorem}
\newtheorem*{theoremA}{Theorem A}
\newtheorem*{theoremB}{Theorem B}

\newtheorem*{theoremD}{Theorem D}
\newtheorem*{theoremE}{Theorem E}

\newtheorem{problem}{Problem}

\newtheorem{proposition}{Proposition}

\theoremstyle{definition}

\newtheorem{remark}{Remark}
\newtheorem{example}{Example}

\usepackage{etoolbox}
\makeatletter
\patchcmd{\@settitle}{\uppercasenonmath\@title}{}{}{}
\patchcmd{\@setauthors}{\MakeUppercase}{}{}{}
\patchcmd{\section}{\scshape}{}{}{}
\makeatother

\begin{document}

\title
[{On a problem of Nathanson on non-minimal additive complements}]
{On a problem of Nathanson on non-minimal additive complements}

\author
[S.--Q. Chen \quad $\&$ \quad Y. Ding]
{Shi--Qiang Chen$^{\dagger}$ \quad {\it and} \quad Yuchen Ding$^{\dagger\dagger}$}

\address{(Shi-Qiang Chen) School of Mathematics and Statistics, Anhui Normal University,  Wuhu  241002,  People's Republic of China}
\email{csq20180327@163.com}
\address{(Yuchen Ding) School of Mathematical Sciences,  Yangzhou University, Yangzhou 225002, People's Republic of China}
\email{ycding@yzu.edu.cn}
\thanks{$^{\dagger}$ ORCID:  0000-0002-0439-6079}
\thanks{$^{\dagger\dagger}$ ORCID:  0000-0001-7016-309X  \quad $\&$ \quad Corresponding author}


\begin{abstract}
Let $C$ and $W$ be two sets of integers. If $C+W=\mathbb{Z}$, then $C$ is called an additive complement to $W$. We further call $C$ a minimal additive complement to $W$ if no proper subset of $C$ is an additive complement to $W$.
Answering a problem of
Nathanson in part, we give sufficient conditions of $W$ which has no minimal additive complements. Our result also extends a prior result of Chen and Yang.
\end{abstract}
\maketitle

\section{Introduction}
Let $\mathbb{Z}$ be the set of all integers. For $C,W\subseteq \mathbb{Z}$ and $c\in \mathbb{Z}$, define
$$
C+W:= \big\{c+w: c\in C,~w\in W\big\} \quad \text{and} \quad cW:=\big\{cw: c\in C\big\}.
$$
For the set $W=\{w\}$ consisted of a single element, we just write $C+\{w\}$ as  $C+w$ for simplicity.
If $C+W=\mathbb{Z}$, then $C$ is called an additive complement to $W$.
Furthermore, if any proper subset of $C$ is not an additive complement to $W$, then $C$ is called a minimal additive complement to $W$. Sometimes, for simplicity, a minimal additive complement is just called {\it a minimal complement} throughout this article.

In 2011, Nathanson \cite[Theorem 8]{N2011} proved the following interesting theorem.
\begin{theorem0}\label{th0} Let $W$ be a nonempty, finite set of integers. In $\mathbb{Z}$, every complement to $W$ contains a minimal complement to $W$.
\end{theorem0}

Nathanson \cite[Problem 11]{N2011} also posed the following open problem.
\begin{problem}\label{pro1} Let $W$ be an infinite set of integers. Does there exist a minimal complement to $W$? Does there exist a complement to $W$ that does not contain a minimal complement?
\end{problem}
In 2012, Chen and Yang \cite[Theorem 1]{CY2012} considered Problem \ref{pro1}. For the situation $\inf W=-\infty$ and $\sup W=+\infty$, they made the following progress.
\begin{theoremA}[Chen--Yang] Let $W$ be a set of integers with $\inf W=-\infty$ and $\sup W=+\infty$. Then there exists a minimal complement to $W$.
\end{theoremA}

It therefore remains to deal with the case $\inf W>-\infty$ or $\sup W<+\infty$, thanks to  {\bf Theorem A}. Without loss of generality, we may assume that $\inf W>-\infty$ since otherwise we could consider $-W$ instead of $W$. Suppose now that $\inf W=w_{0}>-\infty$, then we can consider the set
$$
W-(w_{0}-1):=\{w-(w_{0}-1):w\in W\}
$$
instead of $W$. Thus, without loss of generality, we may further assume that $\inf W=1$.
Based on these observations, Chen and Yang \cite[Theorem 2]{CY2012} proved the following theorem.

\begin{theoremB}[Chen--Yang]\label{THMB} Let $W=\{1=w_1<w_2<\cdots\}$ be a set of integers and
$$\overline{W}=\mathbb{Z}^{+}\setminus W=\{\overline{w_1}<\overline{w_2}<\cdots\},$$
where $\mathbb{Z}^{+}$ is the set of all positive integers as usual.\\
\rm{($a$)} Suppose that
$$
\limsup_{i\rightarrow+\infty}(w_{i+1}-w_i)=+\infty,
$$
then there exists a minimal complement to $W$.\\
\rm{($b$)} Suppose that
$$
\lim_{i\rightarrow+\infty}(\overline{w_{i+1}}-\overline{w_i})=+\infty,
$$
then there does not exist a minimal complement to $W$.
\end{theoremB}
In 2019, Kiss, S\'{a}ndor and Yang \cite{KSY2019} pointed that if
$$
W=\bigcup\limits_{k=0}^{+\infty} [10^k,2\times 10^k],
$$
then it clear that
$$
\limsup_{i\rightarrow+\infty}(w_{i+1}-w_i)=+\infty \quad \text{and} \quad \limsup_{i\rightarrow+\infty}(\overline{w_{i+1}}-\overline{w_i})=+\infty.
$$
Hence the criterion $\lim_{i\rightarrow+\infty}(\overline{w_{i+1}}-\overline{w_i})=+\infty$ in {\bf Theorem B}($b$) cannot be easily improved to $\limsup_{i\rightarrow+\infty}(\overline{w_{i+1}}-\overline{w_i})=+\infty$ since otherwise there will be a contradiction between {\bf Theorem B}($a$) and {\bf Theorem B}($b$) for this $W$. Therefore, it would be meaningful to think of the following question asked by Yang (private communications).
\begin{problem}\label{pro2}Let $W=\{1=w_1<w_2<\cdots\}$ be a set of integers and
$$
\overline{W}=\mathbb{Z}^+\setminus W=\{\overline{w_1}<\overline{w_2}<\cdots\}.
$$
Determine the structures of $W$ with $\limsup_{i\rightarrow+\infty}(\overline{w_{i+1}}-\overline{w_i})=+\infty$ which does not have a minimal complement.
\end{problem}

Extending {\bf Theorem B}($b$), we answer Problem \ref{pro2} by the following theorem which can be viewed as a partial solution to Problem \ref{pro1}.

\begin{theorem}\label{thm1} Let $W=\{1=w_1<w_2<\cdots\}$ be a set of integers and
$$
\overline{W}=\mathbb{Z}^+\setminus W=\{\overline{w_1}<\overline{w_2}<\cdots\}.
$$
If there exists a sequences $i_1<i_2<\cdots$ such that
$$
\lim_{t\rightarrow+\infty}\left(\overline{w_{i_t+1}}-\overline{w_{i_t}}\right)=+\infty \quad
\text{and} \quad
\limsup_{t\rightarrow+\infty}\big|\overline{W}\cap(\overline{w_{i_{t}}},\overline{w_{i_{t+1}}})\big|<+\infty,
$$
then there does not exist a minimal complement to $W$.
\end{theorem}

It would be of interest to show via the following example that our new theorem has some advantages comparing with {\bf Theorem B}.

\begin{example}
Let
$$
W=\{1\}\cup \bigcup\limits_{k=4}^{+\infty}\big[2^k+9,2^{k+1}\big)=\{1=w_1<w_2<\cdots\}.
$$
So, it is plain that
$$
\overline{W}=\mathbb{Z}^+\setminus W=\bigcup\limits_{k=1}^{+\infty}\big[2^k,2^k+8\big]=\{\overline{w_1}<\overline{w_2}<\cdots\}.
$$
Then, clearly we have
$$
\limsup_{i\rightarrow+\infty}(w_{i+1}-w_i)\le 10 \quad \text{and} \quad
\liminf_{i\rightarrow+\infty}(\overline{w_{i+1}}-\overline{w_i})\le 1.
$$
Thus, both {\bf Theorem B}($a$) and {\bf Theorem B}($b$) fail to determine whether such a $W$ has a minimal additive complement. However, it is easily to see that
$$
\lim_{k\rightarrow\infty}\Big(2^{k+1}-\big(2^k+8\big)\Big)=+\infty
$$
and
$$
\lim_{k\rightarrow\infty}\big|\overline{W}\cap\big(2^k+8,2^{k+1}+8\big)\big|=8.
$$
Therefore, we know that $W$ does not have a minimal additive complement, thanks to the criterion from Theorem \ref{thm1} with $\overline{w_{i_{t}}}=2^k+8$ and $\overline{w_{i_{t+1}}}=2^{k+1}+8$.

\end{example}

\begin{remark}
We now give a simple explanation that the criterion given in Theorem \ref{thm1} is essentially optimal. We first note that
$$
\limsup_{i\rightarrow+\infty}(\overline{w_{i+1}}-\overline{w_i})\ge \lim_{t\rightarrow+\infty}\overline{w_{i_t+1}}-\overline{w_{i_t}}=+\infty
$$
for those $W$ assumed in Theorem \ref{thm1} which are consistent with the requirement of Problem \ref{pro2}. We next show that the condition
$$
\limsup_{t\rightarrow+\infty}\big|\overline{W}\cap(\overline{w_{i_{t}}},\overline{w_{i_{t+1}}})\big|<+\infty
$$
cannot be neglected.
To see this, let $W=\bigcup\limits_{k=0}^{+\infty} [10^k,2\times 10^k]$. It is clear that for this $W$ there exists a sequences $i_1<i_2<\cdots$ with $\overline{w_{i_t}}=10^{t+1}-1$ such that
$$
\lim_{t\rightarrow+\infty}\left(\overline{w_{i_t+1}}-\overline{w_{i_t}}\right)=+\infty \quad  \text{and} \quad \limsup_{t\rightarrow+\infty} \big|\overline{W}\cap(\overline{w_{i_{t}}},\overline{w_{i_{t+1}}})\big|=+\infty.
$$
Then there exists a minimal additive complement to $W$ by {\bf Theorem B}($a$) since  $\limsup_{i\rightarrow+\infty}(w_{i+1}-w_i)=+\infty$.
\end{remark}

Along the same line, Kiss, S\'{a}ndor and Yang \cite{KSY2019} investigated certain sets $W$ with
$$
\limsup_{i\rightarrow+\infty}(w_{i+1}-w_i)<+\infty
$$
which have a minimal complement. Precisely, let $X$ be a set of integers. If there exists a positive integer $T$ such that $x+T\in X$ for all sufficiently large integers $x\in X$, then $X$ is said to be eventually periodic with period $T$. If $W$ is eventually periodic with $|\mathbb{N}\setminus W|=+\infty$, then there is some integer $T_W$ so that $w+T_W\in W$ for sufficiently large $w\in W$, which means that
$$
\limsup_{i\rightarrow+\infty}(w_{i+1}-w_i)\le T_W<+\infty.
$$
So, Kiss, S\'{a}ndor and Yang focused on eventually periodic sets. Let $m$ be a positive integer and $A$ a set. We appoint that
$$
A \pmod{m}=\big\{i\in [0,m-1]:i\equiv a\pmod{m} ~\text{for~some}~a\in A\big\}.
$$
Now, suppose that $W$ is an eventually periodic set with period $m$.
By shifting a number, we may assume, without loss generality, that $W$ has the following structure:
\begin{equation}
\label{eq1}W=(m\mathbb{N}+X_m)\cup Y^{(0)}\cup Y^{(1)},
\end{equation}
where $X_m\subseteq \{0,1,\cdots,m-1$\}, $Y^{(0)}\subseteq \mathbb{Z}^{-}$, $Y^{(1)}$ are finite sets with $Y^{(0)}\pmod{m}\subseteq X_m$ and $\big(Y^{(1)}\pmod{m}\big)\cap X_m=\emptyset$. Here, $\mathbb{Z}^{-}$ represents the set of all negative integers. Kiss, S\'{a}ndor and Yang \cite[Theorems 1 and 2]{KSY2019} gave a  necessary and sufficient condition, respectively, for eventually periodic sets $W$ which have a minimal additive complement.

\begin{theoremD}[Kiss--S\'andor--Yang]\label{thD}
Let $W$ be defined in \eqref{eq1}. If there exists  a minimal complement to $W$, then there exists some $C\subseteq \{0,1,\cdots,m-1\}$ such that the following two statements hold:\\
\rm{($a$)} $C+\big(X_m\cup Y^{(1)}\big)\pmod{m}=\{0,1,\cdots,m-1\}$;\\
\rm{($b$)} For any given $c\in C$, there exists some $y\in Y^{(1)}$ such that
$$
c+y\not\equiv c'+x\pmod{m},
$$
where $c'\in C$ and $x\in X_m$.
\end{theoremD}

\begin{theoremE}[Kiss--S\'andor--Yang]\label{thE}
Let $W$ be defined in \eqref{eq1}. Suppose that there exists some $C\subseteq \{0,1,\cdots,m-1\}$ such that the following two conditions hold:\\
\rm{($a$)} $C+\big(X_m\cup Y^{(1)}\big)\pmod{m}=\{0,1,\cdots,m-1\}$;\\
\rm{($b$)} For any given $c\in C$, there exists some $y\in Y^{(1)}$ such that
$$
c+y\not\equiv c'+x\pmod{m},
$$ where $c'\in C\setminus\{c\}$ and $x\in X_m\cup Y^{(1)}$.

Then there exists a minimal complement to $W$.
\end{theoremE}

For other related results about minimal additive complements, one may also refer to the works \cite{AL2021,BS2021,BL2023,CF2011,D1935,KSY2019,K2019,R2001,W1975,Z2023}.

Our second aim in this article is a new result on the situation 
$$
\limsup_{i\rightarrow+\infty}(\overline{w_{i+1}}-\overline{w_i})<+\infty
$$
(comparing this with {\bf Theorem B}(b)). Following \eqref{eq1}, we write
\begin{align}\label{eq-new1}
W=(m\mathbb{N}+X_m)\cup Y^{(0)}\cup Y^{(1)},
\end{align}
where $X_m\subseteq \{0,1,\cdots,m-1$\}, $Y^{(0)}\subseteq \mathbb{Z}^{-}$ is a finite set with $Y^{(0)}\pmod{m}\subseteq X_m$ and $Y^{(1)}$ is a set with $\big(Y^{(1)}\pmod{m}\big)\cap X_m=\emptyset$. So, we would have
$\limsup_{i\rightarrow+\infty}(\overline{w_{i+1}}-\overline{w_i})<+\infty$
by this setting. It is worth mentioning that in our setting here the set $Y^{(1)}$ does not need to be finite which is slightly more general than the ones considered by Kiss, S\'andor and Yang above. Now, we state below our second theorem.

\begin{theorem}\label{newtheorem} Let $W=(m\mathbb{N}+X_m)\cup Y^{(0)}\cup Y^{(1)}$ be as in \eqref{eq-new1} and $k$ a positive integer. If $Y^{(1)}=D+mk\mathbb{N}$ with $D\subseteq[0,m-1]$, then the set $W$ does not have a minimal additive complement.
\end{theorem}

Note that $Y^{(1)}$ in Theorem \ref{newtheorem} is not finite. So, it can not be easily decided whether $W$ has a minimal additive complement by {\bf Theorem D} and {\bf Theorem E}.  

The paper will be organized as follows: In section 2, we prove Theorem \ref{thm1}. In section 3, we first give a proposition and then deduce Theorem \ref{newtheorem} from it.

\section{Proof of Theorem \ref{thm1}}
\begin{proof}[Proof of Theorem \ref{thm1}]
Suppose the contrary, we assume that $C$ is a minimal additive complement to $W$. Then we clearly have $\inf C=-\infty$. Let
$$C\cap (-\infty,-\overline{w_{i_1+1}})=\{c_1>c_2>\cdots\}.$$
For any $s\in \mathbb{Z}^+$, we assume that
\begin{align}\label{eq2-1}
\overline{w_{i_{u_s-1}+1}}\leq -c_s<\overline{w_{i_{u_s}+1}}.
\end{align}
Since $\limsup_{t\rightarrow+\infty}\big|\overline{W}\cap(\overline{w_{i_{t}}},\overline{w_{i_{t+1}}})\big|<+\infty$, there is an absolute constant $K$ so that
$$
\limsup_{t\rightarrow+\infty}\big|\overline{W}\cap(\overline{w_{i_{t}}},\overline{w_{i_{t+1}}})\big|=K,
$$
which means that there is some $t_W$ so that for any $t\ge t_W$ we have
\begin{align}\label{eq2-2}
\big|\overline{W}\cap(\overline{w_{i_{t}}},\overline{w_{i_{t+1}}})\big|\le K.
\end{align}
Our proof will be separated into the following three cases.

{\bf Case 1.} $\liminf_{s\rightarrow+\infty} (-c_s-\overline{w_{i_{u_s-1}+1}})<+\infty$. Let
$$
h:=\liminf_{s\rightarrow+\infty} (-c_s-\overline{w_{i_{u_s-1}+1}}).
$$
Then $0\le h<+\infty$. Define
$$
H:=\{s:-c_s-\overline{w_{i_{u_s-1}+1}}=h\}.
$$
It is clear that $H$ is an infinite set since $-c_s-\overline{w_{i_{u_s-1}+1}}$ are all integers.
Let $s_0\in H$ be any given integer. Then there exist integers $n_{s_0}\in \mathbb{Z}$ and $w_{s_0}\in W$ such that
\begin{align}\label{eq2-3}
n_{s_0}=c_{s_0}+w_{s_0} \quad \text{but} \quad n_{s_0}\neq c+w
\end{align}
for any $w\in W$ and $c\in C$ with $c\neq c_{s_0}$,
otherwise $C\setminus\{c_{s_0}\}$ will be an additive complement to $W$ which contradicts with the minimal property of $C$.
Recall that $\inf C=-\infty$ and
$$
\lim_{t\rightarrow+\infty}\left(\overline{w_{i_t+1}}-\overline{w_{i_t}}\right)=+\infty
$$
from the condition of our theorem. There exists an integer $s_0^*> s_0$ such that
\begin{align}\label{eq2-4}
n_{s_0}-c_{s}>0
\end{align}
for all $s\in H$ with $s\ge s_0^*$ and
\begin{align}\label{eq2-5}
\overline{w_{i_t+1}}-\overline{w_{i_t}}>|n_{s_0}|+h
\end{align}
for all $t\geq u_{s_0^*}-1$.
From (\ref{eq2-3}) and (\ref{eq2-4}), we have
\begin{align}\label{eq2-6}
n_{s_0}-c_{s}\in \overline{W}
\end{align}
for $s\ge  s_0^*$. From (\ref{eq2-5}) with $t=u_s$ we have
\begin{align*}
\overline{w_{i_{u_s}+1}}-\overline{w_{i_{u_s}}}>|n_{s_0}|+h=|n_{s_0}|-c_s-\overline{w_{i_{u_s-1}+1}}\geq|n_{s_0}|-c_s-\overline{w_{i_{u_s}}}£¬
\end{align*}
providing that $s\ge s_0^*$.
In other words,
\begin{align}\label{eq2-7}
n_{s_0}-c_s\leq |n_{s_0}|-c_s<\overline{w_{i_{u_s}+1}}.
\end{align}
On the other hand, we know again from (\ref{eq2-5}) with $t=u_s-1$ that
\begin{align}\label{eq2-8}
\overline{w_{i_{u_s-1}+1}}-\overline{w_{i_{u_s-1}}}>|n_{s_0}|+h\geq -n_{s_0}.
\end{align}
So by (\ref{eq2-1}) and (\ref{eq2-8}), we have
\begin{align}\label{eq2-9}
n_{s_0}-c_s\ge n_{s_0}+\overline{w_{i_{u_{s}-1}+1}}>\overline{w_{i_{u_{s}-1}}},
\end{align}
provided that $s\ge s_0^*$.
Hence, we conclude from (\ref{eq2-6}), (\ref{eq2-7}) and (\ref{eq2-9}) that
\begin{align}\label{eq2-10}
\overline{w_{i_{u_{s}-1}}}<n_{s_0}-c_s<\overline{w_{i_{u_s}+1}} \quad \text{and} \quad n_{s_0}-c_{s}\in \overline{W}
\end{align}
for all $s_0,s\in H$ with $s\ge s_0^*> s_0$.

Now, we chose $K+2$ different integers $s_{0,1},s_{0,2},...,s_{0,K+2}\in H$. Then there are $K+2$ integers $s_{0,1}^*,s_{0,2}^*,...,s_{0,K+2}^*$ so that for any $1\le j\le K+2$ we have
\begin{align}\label{eq2-11}
\overline{w_{i_{u_{s}-1}}}<n_{s_{0,j}}-c_s<\overline{w_{i_{u_s}+1}} \quad \text{and} \quad n_{s_{0,j}}-c_{s}\in \overline{W} \quad\quad  (\forall ~s\ge s_{0,j}^*)
\end{align}
from (\ref{eq2-10}). Now let $s$ be a sufficiently large integer so that
$$s>\max\left\{s_{0,1}^*,s_{0,2}^*,...,s_{0,K+2}^*\right\}$$
and $u_s-1\ge t_W$. Then for such $s$ we have
$$
n_{s_{0,j}}-c_s\in \overline{W}\cap(\overline{w_{i_{u_s-1}}},\overline{w_{i_{u_s}+1}}) \quad \quad (\forall ~1\le j\le K+2)
$$
from (\ref{eq2-11}). Moreover, we know from (\ref{eq2-2}) that
\begin{align*}
\big|\overline{W}\cap(\overline{w_{i_{u_s-1}}},\overline{w_{i_{u_s}+1}})\big|
= \big|\overline{W}\cap(\overline{w_{i_{u_s-1}}},\overline{w_{i_{u_s}}})\big|+\big|\overline{W}\cap[\overline{w_{i_{u_s}}},\overline{w_{i_{u_{s}}+1}})\big|
\le K+1,
\end{align*}
from which it follows that there at least two integers
$$
1\le j_1<j_2\le K+2
$$
such that
$$
n_{s_{0,{j_1}}}-c_s=n_{s_{0,{j_2}}}-c_s, \quad \text{i.e.,} \quad n_{s_{0,{j_1}}}=n_{s_{0,{j_2}}}.
$$
By (\ref{eq2-3}), we clearly have
$$
n_{s_{0,{j_1}}}=c_{s_{0,{j_1}}}+w_{s_{0,{j_1}}} \quad \text{and} \quad n_{s_{0,{j_2}}}=c_{s_{0,{j_2}}}+w_{s_{0,{j_2}}},
$$
which is certainly a contradiction with (\ref{eq2-3}) since $n_{s_{0,{j_1}}}=n_{s_{0,{j_2}}}$ whereas $c_{s_{0,{j_1}}}\neq c_{s_{0,{j_2}}}$.

{\bf Case 2.} $\liminf_{s\rightarrow+\infty} (c_s+\overline{w_{i_{u_s}+1}})<+\infty$. The proof is similar to {\bf Case 1}. Let
$$
\ell:=\liminf_{s\rightarrow+\infty} (c_s+\overline{w_{i_{u_s}+1}}).
$$
Then $0\le \ell <+\infty$. Define
$$
L:=\{s:c_s+\overline{w_{i_{u_s}+1}}=\ell\}.
$$
It is clear that $L$ is an infinite set since $c_s+\overline{w_{i_{u_s}+1}}$ are all integers.
Let $s_0\in H$ be any given integer. Then there exist integers $n_{s_0}\in \mathbb{Z}$ and $w_{s_0}\in W$ such that
\begin{align}\label{eq2-12}
n_{s_0}=c_{s_0}+w_{s_0} \quad \text{but} \quad n_{s_0}\neq c+w
\end{align}
for any $w\in W$ and $c\in C$ with $c\neq c_{s_0}$,
otherwise $C\setminus\{c_{s_0}\}$ will be an additive complement to $W$ which contradicts with the minimal property of $C$. Since $\inf C=-\infty$ and
$$
\lim_{t\rightarrow+\infty}\left(\overline{w_{i_t+1}}-\overline{w_{i_t}}\right)=+\infty
$$
from the condition of our theorem, there exists an integer $s_0^*> s_0$ such that
\begin{align}\label{eq2-13}
n_{s_0}-c_{s}>0
\end{align}
for all $s\in L$ with $s\ge s_0^*$ and
\begin{align}\label{eq2-14}
\overline{w_{i_t+1}}-\overline{w_{i_t}}>|n_{s_0}|+\ell
\end{align}
for all $t\geq u_{s_0^*}$. From (\ref{eq2-12}) and (\ref{eq2-13}), we have
\begin{align}\label{eq2-15}
n_{s_0}-c_{s}\in \overline{W}
\end{align}
for $s\ge  s_0^*$. By (\ref{eq2-14}) with $t=u_s+1$, we have
\begin{align}\label{eq2-16}
\overline{w_{i_{u_s+1}+1}}-\overline{w_{i_{{u_s}+1}}}>|n_{s_0}|+\ell\geq n_{s_0},
\end{align}
providing that $s\ge s_0^*$. From (\ref{eq2-1}) and (\ref{eq2-16}) for $s\ge s_0^*$ we have
\begin{align}\label{eq2-17}
n_{s_0}-c_s<n_{s_0}+\overline{w_{i_{u_s}+1}}\leq n_{s_0}+\overline{w_{i_{u_s+1}}}<\overline{w_{i_{u_s+1}+1}}.
\end{align}
On the other hand, using (\ref{eq2-16}) again we get
$$
\overline{w_{i_{u_s}+1}}-\overline{w_{i_{u_s}}}>|n_{s_0}|+\ell= |n_{s_0}|+c_s+\overline{w_{i_{u_s}+1}},
$$
by the definition of $\ell$, from which it follows that
\begin{align}\label{eq2-18}
n_{s_0}-c_s \geq -|n_{s_0}|-c_s>\overline{w_{i_{u_{s}}}},
\end{align}
provided that $s\ge s_0^*$. Hence, we conclude from (\ref{eq2-15}), (\ref{eq2-17}) and (\ref{eq2-18}) that
\begin{align}\label{eq2-19}
\overline{w_{i_{u_{s}}}}<n_{s_0}-c_s<\overline{w_{i_{u_s+1}+1}} \quad \text{and} \quad n_{s_0}-c_{s}\in \overline{W}
\end{align}
for all $s_0,s\in L$ with $s\ge s_0^*> s_0$.

Now, the remaining arguments are the same to the last paragraph of {\bf Case 1} , so we just omit them.

{\bf Case 3.} $
\liminf_{s\rightarrow+\infty} (-c_s-\overline{w_{i_{u_s-1}+1}})=
\liminf_{s\rightarrow+\infty} (c_s+\overline{w_{i_{u_s}+1}})=+\infty.
$
Since $C$ is a minimal complement to $W$, then for any given integer $s_0\geq1$ there exist integers $n_{s_0}\in \mathbb{Z}$ and $w_{s_0}\in W$ such that
\begin{align}\label{eq2-20}
n_{s_0}=c_{s_0}+w_{s_0} \quad \text{but} \quad n_{s_0}\neq c+w
\end{align}
for any $w\in W$ and $c\in C$ with $c\neq c_{s_0}$. By conditions of {\bf Case 3},
there exist infinitely many integers $s>s_0$ such that
$$
-c_{s}-\overline{w_{i_{u_s-1}+1}}>|n_{s_0}|
\quad \text{and} \quad  c_s+\overline{w_{i_{u_s}+1}}>|n_{s_0}|,
$$
which implies that
\begin{align}\label{eq2-21}
n_{s_0}-c_s\leq |n_{s_0}|-c_s<\overline{w_{i_{u_s}+1}} \quad
\text{and} \quad n_{s_0}-c_s\geq -|n_{s_0}|-c_s>\overline{w_{i_{u_s-1}+1}}.
\end{align}
Hence, we conclude from (\ref{eq2-20}) and (\ref{eq2-21}) that
\begin{align*}
\overline{w_{i_{u_{s}-1}+1}}<n_{s_0}-c_s<\overline{w_{i_{u_s}+1}} \quad \text{and} \quad n_{s_0}-c_{s}\in \overline{W}
\end{align*}
for any $s>s_0$.

Using similar arguments of {\bf Case 1}, the validity of the last case clearly follows.
\end{proof}

\section{Proof of Theorem \ref{newtheorem}}
Theorem \ref{newtheorem} will be deduced from the following proposition.

\begin{proposition}\label{thm2}  Let $W=(m\mathbb{N}+X_m)\cup Y^{(0)}\cup Y^{(1)}$ be as in \eqref{eq-new1}. Suppose that $S$ is an integer set satisfying the condition that
for any integer set $G$ with property
$$
(-\infty,n_G]\subseteq G+S
$$
for some integer $n_G$ (probably negative), we have
$$
G\setminus\{g\}+S=G+S \quad (\forall g\in G).
$$
Then for any $Y^{(1)}=D+mS$ with $D\subseteq[0,m-1]$ the set $W$ does not have a minimal additive complement.
\end{proposition}

\begin{proof}
Suppose the contrary, we assume that $C$ is a minimal additive complement to $W$.
We will deduce a contradiction.
For any $i\in[0,m-1]$, we define
$$
C_i=\{c\in C:c\equiv i\pmod{m}\}
$$
and
$$
J=\big\{j: \big|C_j\cap \mathbb{Z}^{-}\big|=+\infty,~0\leq j\leq m-1\big\}.
$$
For any integer (if there does exist) $t\in [0,m-1]\setminus J$, we have $|C_t\cap \mathbb{Z}^{-}|<+\infty$, and so
$$\inf\bigg(\bigcup\limits_{t\in [0,m-1]\setminus J} (C_t+W)\bigg)>-\infty,$$
which implies that
$$J+W\pmod{m}=[0,m-1].$$
We define $\mathcal{C}_1=J+X_m\pmod{m}$ and $\mathcal{C}_2=J+Y^{(1)}\pmod{m}$. Then
$$
\mathcal{C}_1\cup \mathcal{C}_2=J+\big(X_m\cup Y^{(1)}\big)\pmod{m}=J+W\pmod{m}=[0,m-1],
$$
where the last but one equality follows from the fact $Y^{(0)}\pmod{m}\subseteq X_m$. The arguments will be separated into the following two cases.

{\bf Case 1.} $\mathcal{C}_2\subseteq \mathcal{C}_1$.
Let $c_0\in C$ be any fixed number. Noting that $\mathcal{C}_1\cup \mathcal{C}_2=[0,m-1]$, we clearly have $\mathcal{C}_1=[0,m-1]$. So, for any integer $n$, there exist $k\in \mathbb{Z}$, $x\in X_m$ and $j\in J$ such that $n=j+x+km$. Recall that $|C_j\cap \mathbb{Z}^{-}|=+\infty$ for any $j\in J$. There exists a positive integer $t>|k|$ such that $j-tm\in C$. Thus,
 $$n=(j-tm)+\big(x+(k+t)m\big)\in C+W.$$
If $j-tm\neq c_0$, then $n\in C\setminus\{c_0\}+W$. If $j-tm=c_0$, then there exists a positive integer $t'>t$ such that $j-(t+t')m\in C$ since $|C_j\cap \mathbb{Z}^{-}|=+\infty$ for any $j\in J$.  Thus,
$$
n=(j-(t+t')m)+(x+(k+t+t')m)\in C\setminus\{c_0\}+W.
$$
Hence, $C\setminus\{c_0\}+W=\mathbb{Z}$ which is a contradiction with the minimal property of $C$. 
 
{\bf Case 2.} $\mathcal{C}_2\not\subseteq \mathcal{C}_1$. For any $i\in[0,m-1]$ and $H,K\subseteq[0,m-1]$, we define
$$
C(i;H,K)=\big\{(h,k):h+k\equiv i \pmod{m},~h\in H,~k\in K\big\}.
$$
For any $(h,k)\in C(i;H,K)$ let
$$
\mathcal{A}_{h,k}=\big((h+m\mathbb{Z})\cap C\big)+\big((k+m\mathbb{Z})\cap W\big).
$$
Let $\mathscr{C}=C\pmod{m}$ and $\mathscr{W}=W\pmod{m}$. Then for any $i\in[0,m-1]$, we claim
\begin{align}\label{eqth2-1}
i+m\mathbb{Z}&=\bigcup\limits_{(c,w)\in C(i,\mathscr{C},\mathscr{W})}\mathcal{A}_{c,w}.
\end{align}
In fact, for any $(c,w)\in C(i,\mathscr{C},\mathscr{W})$ it is plain that
\begin{align*}
\mathcal{A}_{c,w}=((c+m\mathbb{Z})\cap C\big)+\big((w+m\mathbb{Z})\cap W\big)\subseteq (c+m\mathbb{Z})+(w+m\mathbb{Z})\subseteq i+m\mathbb{Z}.
\end{align*}
On the other hand, for any $i+mt\in i+m\mathbb{Z}$, we can assume
$$
i+mt=c+w \quad (c\in C, w\in W)
$$
since $C+W=\mathbb{Z}$. We now suppose that 
$
c\equiv c_1\pmod{m}$ 
and 
$
 w\equiv w_1\pmod{m},
$
then 
$$
c_1+w_1\equiv c+w\equiv i\pmod{m}.
$$
Thus, we have
$
i+mt\in \mathcal{A}_{c_1,w_1}.
$
Note that $W=(m\mathbb{N}+X_m)\cup Y^{(0)}\cup Y^{(1)}$ with $Y^{(0)}$ being a finite set and $Y^{(0)}\pmod{m}\subseteq X_m$. We have $\mathscr{W}=X_m\cup D\pmod{m}$. So, we deduce
\begin{align}\label{eqth2-2}
\bigcup\limits_{(c,w)\in C(i,\mathscr{C},\mathscr{W})}\mathcal{A}_{c,w}=\bigg(\bigcup\limits_{(c,w)\in C(i,\mathscr{C},X_m)}\mathcal{A}_{c,w}\bigg)\cup\bigg(\bigcup\limits_{(c,d)\in C(i,\mathscr{C},D)}\mathcal{A}_{c,d}\bigg).
\end{align}
Trivially, we have $\mathscr{C}=J\cup(\mathscr{C}\setminus J)$. Hence, we conclude from (\ref{eqth2-1}) and (\ref{eqth2-2}) that
\begin{align}\label{eqth2-3}
i+m\mathbb{Z}=\!\bigg(\bigcup\limits_{(c,w)\in C(i,J,X_m)}\!\mathcal{A}_{c,w}\bigg)\cup\bigg(\bigcup\limits_{(c,w)\in C(i,\mathscr{C}\setminus J,X_m)}\!\mathcal{A}_{c,w}\bigg)\cup\bigg(\bigcup\limits_{(c,d)\in C(i,\mathscr{C},D)}\!\mathcal{A}_{c,d}\bigg).
\end{align}
 
For any $d\in D$, it clear that 
$$
(d+m\mathbb{Z})\cap W=d+mS
$$
since $D\pmod{m}=Y^{(1)}\pmod{m}$, $Y^{(0)}\pmod{m}\subseteq X_m$ and $\big(Y^{(1)}\pmod{m}\big)\cap X_m=\emptyset$.
For any $(c,d)\in C(i,\mathscr{C},D)$, we have $c+d\equiv i \pmod{m}$. Suppose that $c+d=i+mk_{c,d}$ and $C\cap (c+m\mathbb{Z})=c+mG_{c}$ for $c\in \mathscr{C}$,
then
\begin{align}\label{eqth2-4}
\bigcup\limits_{(c,d)\in C(i,\mathscr{C},D)}\mathcal{A}_{c,d}&=\bigcup\limits_{(c,d)\in C(i,\mathscr{C},D)}\Big(\big((c+m\mathbb{Z})\cap C\big)+(d+mS)\Big)\nonumber\\
&=\bigcup\limits_{(c,d)\in C(i,\mathscr{C},D)} \Big((c+mG_{c})+(d+mS)\Big)\nonumber\\
&=\bigcup\limits_{(c,d)\in C(i,\mathscr{C},D)}\Big((c+d)+m(G_{c}+S)\Big)\nonumber\\
&=\bigcup\limits_{(c,d)\in C(i,\mathscr{C},D)}\Big(i+mk_{c,d}+m(G_{c}+S)\Big)\nonumber\\
&=\bigcup\limits_{(c,d)\in C(i,\mathscr{C},D)}\Big(i+m\big((k_{c,d}+G_{c})+S\big)\Big)\nonumber\\
&=i+m\bigg(\Big(\bigcup\limits_{(c,d)\in C(i,\mathscr{C},D)}(k_{c,d}+G_{c})\Big)+S\bigg).
\end{align}

Taking $c_0\in C$ with $c_0\pmod{m}\in J$. Noting that $C$ is a minimal additive complement to $W$, there exist integers $n_0\in \mathbb{Z}$ and $w_0\in W$ such that 
\begin{align}\label{new-1}
n_0=c_0+w_0 \quad \text{and} \quad n_0\neq c+w
\end{align}
for any $w\in W$ and $c\in C\setminus\{c_0\}$. The remaining argument will be separated into the following two subcases.

{\bf Subcase 2.1.} If $n_0\pmod{m}\in \mathcal{C}_1\pmod{m}$, then 
$$
n_0\pmod{m}\in J+X_m\pmod{m}
$$ 
by the definition of $\mathcal{C}_1$.
Hence,  there exist $k\in \mathbb{Z}$, $x\in X_m$ and $j\in J$ such that 
$$
n_0=j+x+km.
$$
Noting that $|C_j\cap \mathbb{Z}^{-}|=+\infty$ for any $j\in J$, there exists an positive
integer $t>|k|$ such that $j-tm\in C$. Thus,
$$n_0=(j-tm)+(x+(k+t)m)\in C+W.$$
If $j-tm\neq c_0$, then $n_0\in C\setminus\{c_0\}+W$. If $j-tm=c_0$, then by $|C_j\cap \mathbb{Z}^{-}|=+\infty$ again, there exists an positive integer $t'>t$ such that $j-(t+t')m\in C$.  Then
$$
n_0=(j-(t+t')m)+(x+(k+t+t')m)\in C\setminus\{c_0\}+W,
$$
which leads to a contradiction.

{\bf Subcase 2.2.} If $n_0\pmod{m}\not\in \mathcal{C}_1\pmod{m}$, then there exists $i$ with $i\in \mathcal{C}_2\setminus \mathcal{C}_1$ such that $n_0\equiv i\pmod{m}$ since $\mathcal{C}_1\cup \mathcal{C}_2=[0,m-1]$. It now follows from \eqref{eqth2-3} that
\begin{align}\label{eqth2-5}
i+m\mathbb{Z}=\bigg(\bigcup\limits_{(c,w)\in C(i,\mathscr{C}\setminus J,X_m)}\!\mathcal{A}_{c,w}\bigg)\cup\bigg(\bigcup\limits_{(c,d)\in C(i,\mathscr{C},D)}\!\mathcal{A}_{c,d}\bigg).
\end{align}
Noting that $|C_j\cap \mathbb{Z}^{-}|<+\infty$ for any $j\in \mathscr{C}\setminus J$, there exists an integer $a$ such that
\begin{align}\label{eqth2-6}
\bigcup\limits_{(c,w)\in C(i,\mathscr{C}\setminus J,X_m)}\mathcal{A}_{c,w}\subseteq i+m[a,+\infty),
\end{align}
which implies from \eqref{eqth2-4}, \eqref{eqth2-5} and \eqref{eqth2-6} that
\begin{align}\label{eqth2-7}
(-\infty, n_G]\subseteq \bigg(\bigcup\limits_{(c,d)\in C(i,\mathscr{C},D)}(k_{c,d}+G_{c})\bigg)+S
\end{align}
for some integer $n_G$.
Let 
$
\widetilde{G}_i=\bigcup\limits_{(c,d)\in C(i,\mathscr{C},D)}(k_{c,d}+G_{c}).
$ 
Then for any $g\in \widetilde{G}_i$ we have
\begin{align}\label{eqth2-8}
\widetilde{G}_i\setminus\{g\}+S=\widetilde{G}_i+S
\end{align}
thanks to \eqref{eqth2-7} and the condition of our theorem.

For the integer $c_0$ chosen formerly, we assume that
$$
c_0\equiv c'\pmod{m},
$$
where $c'\in \mathscr{C}$.
Then there exists $g'\in G_{c'}$ so that $c_0=c'+mg'$ by the definition of $G_{c'}$. Recall that $n_0=c_0+w_0$ with $c_0\pmod{m}\in J$ and 
$$
n_0\pmod{m}\in \mathcal{C}_2(=J+Y^{(1)})\setminus\mathcal{C}_1\pmod{m}.
$$
So, we can deduce
$
w_0\in Y^{(1)}.
$
Hence, there 
exist an integer $d\in D$ and $s\in S$ with $(c',d)\in C(i,\mathscr{C},D)$ such that $w_0=d+ms$.
Thus, we have
\begin{align}\label{eqth2-10}
n_0=&c_0+w_0\nonumber\\
=&(c'+mg')+(d+ms)\nonumber\\
=&i+mk_{c',d}+m(g'+s)\nonumber\\
=&i+m((k_{c'',d}+g')+s)\nonumber\\
=&i+m(g+s),
\end{align}
where $g=k_{c',d}+g'$.
Noting that $g\in \widetilde{G}_i$, we have $\widetilde{G}_i\setminus\{g\}+S=\widetilde{G}_i+S$ by \eqref{eqth2-8}. Hence, there exist integers $g_1\in \widetilde{G}_i$ with $g_1\neq g$ and $s_1\in S$ such that $g+s=g_1+s_1$. Noting also that $g_1\in \widetilde{G}_i$, there exists integer $c'_1\in \mathscr{C}$, integer $d_1\in D$ and $g'_1\in G_{c_1'}$ such that $g_1=k_{c'_1,d_1}+g'_1$ by the definition of $\widetilde{G}_i$. Therefore, by \eqref{eqth2-10} we have
\begin{align}\label{eqth2-11}
n_0=&i+m(g_1+s_1)\nonumber\\
=&i+m\big((k_{c'_1,d_1}+g'_1)+s_1\big)\nonumber\\
=&i+mk_{c'_1,d_1}+m(g'_1+s_1)\nonumber\\
=&(c'_1+mg'_1)+(d_1+ms_1).
\end{align}
Recall that  $c'_1+mg'_1\in C$
and 
$
n_0\neq c+w
$
for any $w\in W$ and $c\in C\setminus\{c_0\}$ by \eqref{new-1}, it follows from \eqref{eqth2-11} that
$$
c_0=c'+mg'=c'_1+mg'_1,
$$
which implies that $c'=c'_1$ and $g'=g'_1$. Noting further that $g\neq g_1$, we deduce 
$$
g-g_1=k_{c',d}+g'-(k_{c'_1,d_1}+g'_1)=k_{c',d}-k_{c'_1,d_1}\neq 0,
$$
which leads to the fact $d\neq d_1$ because $c'=c'_1$. Since
$$c'+d\equiv c'_1+d_1\equiv i\pmod{m}$$
and $c'=c'_1$, it follows that $d\equiv d_1\pmod{m}$. On observing $d,d_1\in D$, we have $d=d_1$, which is certainly a contradiction.
\end{proof}

We now carry out the proof of Theorem \ref{newtheorem}.

\begin{proof}[Proof of Theorem \ref{newtheorem} via Proposition \ref{thm2}]
Let $G$ be a set satisfying
$$
(-\infty,n_G]\subseteq G+k\mathbb{N}
$$
for some integer $n_G$.
Let $g$ be any fixed element of $G$. We prove that
$$
n\in G\setminus\{g\}+k\mathbb{N} 
$$
for any $n\in G+k\mathbb{N}$, from which it follows that
$$
(G\setminus\{g\})+k\mathbb{N}=G+k\mathbb{N}.
$$
In fact, for any $n\in G+k\mathbb{N}$ there exist integers $g'\in G$ and $n'\in \mathbb{N}$ such that $n=g'+kn'$. If $g'\neq g$, then $n\in G\setminus\{g\}+k\mathbb{N}$.  It remains to consider the case $g'=g$. In this case, we have
$$
n=g+kn'.
$$
Clearly, there exists some positive integer $t$ such that $g-kt<n_G$.
Recall that
$$
(-\infty,n_G]\subseteq G+k\mathbb{N},
$$
so we have $g-kt\in G+k\mathbb{N}$. Hence, $g-kt=g_1+k\ell$ for some positive integer $\ell$ and some $g_1\in G$, i.e., $g-k(t+\ell)\in G$. Thus,
$$
n=g+kn'=\big(g-k(t+\ell)\big)+k(n'+t+\ell)\in G\setminus\{g\}+k\mathbb{N}.
$$
By Proposition \ref{thm2}, there does not exist a minimal complement to $W$.
\end{proof}

\section*{Acknowledgments}
The first named author is supported by National Natural Science Foundation of China (Grant No. 12301003), Anhui Provincial Natural Science Foundation (Grant No. 2308085QA02) and University Natural Science Research Project of Anhui Province (Grant No. 2022AH050171).

The second named author is supported by National Natural Science Foundation of China  (Grant No. 12201544), Natural Science Foundation of Jiangsu Province, China (Grant No. BK20210784), China Postdoctoral Science Foundation (Grant No. 2022M710121).

\end{document}